\documentclass[10pt]{amsart}

\usepackage{amssymb, amsthm, amsmath, graphicx, yfonts, wasysym, appendix, url, verbatim}
\usepackage[margin=4cm]{geometry}
\usepackage{rotating}
\usepackage{amsbsy,enumerate}
\usepackage{graphicx}
\usepackage{ccaption}
\usepackage{xcolor}
\usepackage{times}
\usepackage{wrapfig}
\usepackage{amsmath, amssymb, amsthm} 

\newcommand{\IP}[2]{\left< #1 , #2 \right>}

\newcommand{\R}{\ensuremath{\mathbb{R}}}
\newcommand{\D}{\text{D}}

\allowdisplaybreaks[4]

\newtheorem{thm}{Theorem}[section]
\newtheorem{cor}[thm]{Corollary}
\newtheorem{prop}[thm]{Proposition}
\newtheorem{lem}[thm]{Lemma}

\newtheorem*{uthm}{Theorem}

\newtheorem{defn}[thm]{Definition}

\theoremstyle{remark}
\newtheorem*{rmk}{Remark}

\title[Mean curvature flow in embedded cylinders]{Mean curvature flow with free
boundary in embedded cylinders or cones and uniqueness results for minimal
hypersurfaces}
\author{Valentina-Mira Wheeler}

\address{ Valentina-Mira Wheeler \\
           Institute for Mathematics and its Applications \\
           University of Wollongong\\
           Northfields Avenue\\
           Wollongong, NSW, 2522, Australia\\
           email: vwheeler@uow.edu.au
           }


\begin{document}

\begin{abstract}

In this paper we study the mean curvature flow of embedded disks with free
boundary on an embedded cylinder or generalised cone of revolution, called the
support hypersurface.
We determine regions of the interior of the support hypersurface such that initial data
is driven to a curvature singularity in finite time or exists for all time and
converges to a minimal disk.
We further classify the type of the singularity.
We additionally present applications of these results to the uniqueness problem
for minimal hypersurfaces with free boundary on such suppport hypersurfaces;
the results obtained this way do not require a-priori any symmetry or
topological restrictions.
\end{abstract}

\maketitle
\section{Introduction}
Minimal surfaces and the mean curvature flow in the free boundary setting are
natural extrinsic geometric elliptic and parabolic problems that have appeared
sporadically throughout the literature for some time (see Nitsche \cite{hildebrandt1979minimal,nitsche1985lectures},
Hildebrandt, Dierkes and collaborators \cite{dierkes1992minimal,dierkes2010regularity,dierkes2010global} for historical remarks).
Inspired by work on the closed hypersurfaces by Huisken \cite{huisken1984flow} and on the
Ricci flow by Hamilton \cite{hamilton1982three}, Stahl in 1994 made a fundamental contribution
\cite{thesisstahl}, establishing local and global existence plus blowup
results. Since this time, work has greatly intensified.

We say that a smooth one-parameter family of immersed disks
$F:D^n\times[0,T)\rightarrow\R^{n+1}$ evolves by the mean curvature flow with
free boundary on a support hypersurface $F_\Sigma:\Sigma\rightarrow\R^{n+1}$ if
\begin{align}
\frac{\partial F}{\partial t} = \vec{H} = -H\nu\quad\quad&\text{ on }D^n\times[0,T)
\notag\\
\IP{\nu}{\nu_\Sigma} = 0\quad\quad\quad\quad&\text{ on }\partial D^n\times[0,T)\,,
\label{MCFwFB}\\
F(\partial D^n,t) \subset F_\Sigma(\Sigma)\,,\quad\text{and}&\quad
F(\cdot,0) = F_0(\cdot)\,.
\notag
\end{align}

Local existence follows, as demonstrated by Stahl \cite{thesisstahl}, by
writing the evolving hypersurfaces as graphs
for a short time over their initial data.
Stahl additionally gave continuation criteria: a-priori bounds on the second
fundamental form are sufficient for the global existence of a solution
\cite{stahl1996res,stahl1996convergence}.
In this work he also showed that initially convex data remains convex when the
support hypersurface is umbilic, and that in this situation the flow contracts
to a round hemishperical point (a Type I singularity). A generalisation to
other contact angles of Stahl's continuation criteria was later obtained by
Freire \cite{freire2010mean}.

Buckland studied a setting similar to that of Stahl, and focused on obtaining a
classification of singularities according to topology and type
\cite{buckland2005mcf}.
Koeller has generalised the regularity theory developed by Ecker and Huisken
\cite{ecker2004rtm,ecker1989mce,ecker1991ieh} to the setting of free boundaries
\cite{koeller2007singular}. His main regularity theorem is a criterion under
which the singular set will has measure zero.

The author has studied initially graphical mean curvature flow with free
boundary, obtaining long time existence results and results on the formation of
curvature singularities on the free boundary
\cite{vmwheeler2012rotsym,wheeler2014mean,wheeler2014meanhyperplane}.
A similar angle approach has been employed by Lambert
\cite{lambert2014perpendicular} in his work.
Edelen's work is the first systematic treatment of Type II singularities \cite{edelen2014convexity}.
Convexity estimates play a fundamental role in his work.

Regular solutions of the mean curvature flow with bounded initial area converge
as $t\rightarrow \infty$ to minimal hypersurfaces. This also occurs in the
setting of free boundary, so it is natural to consider the mean curvature flow
as a tool to study minimal surfaces.

Minimal surfaces (and hypersurfaces) are a classical topic in mathematics and as such have received enormous attention in the
literature. A review is well beyond the scope of this paper.
See for example
\cite{osserman2002survey,giusti1984minimal,schoen1970infinite,schoen1979existence,lawson1970complete,choe1990isoperimetric,choe1992sharp,colding2004space,colding2011course,nitsche1965new,andrews2015embedded}
and the references within. The studies are extensive and from many
perspectives: harmonic analysis, geometry, calculus of variations and
isoperimetry, complex analysis, partial differential equations, spectral
theory, and more.

Work in the free boundary setting is also abundant, see for example
\cite{fraser2011first,fraser2012sharp,courant2005dirichlet,jager1970behavior,hildebrandt1979minimal,gruter1981boundary}
and the references therein.
Nevertheless there remain many fundamental open questions, in particular to do
with the classification and uniqueness of minimal surfaces with free boundary.

Uniqueness for surfaces of prescribed mean curvature has been previously
treated by Vogel in \cite{vogel1988uniqueness} under certain conditions.
Minimal surfaces and capillarity surfaces of constant mean curvature in right
solid cylinders and cones have been studied before by Choe--Park, Lopez--Pyo in
\cite{choe2011capillary,lopez2014capillary,lopez2014capillary2} via
geometric and eliptic techniques. The authors have many results in these papers
and others, involving constant mean curvature surfaces with free boundary
that invite flow applications. We hope that we are able to inspire progress
in this direction.

In this paper we apply the mean curvature flow with free boundary to prove a
result in this direction (Theorem \ref{thmuniqueness}).

In particular, we prove uniqueness and non-existence results for minimal
hypersurfaces supported on oscillating or pinching cylinders (embedded double
cones) in Euclidean space.
There are no dimension, topological, or symmetry restrictions on our results.
For example, we prove:

\begin{uthm}
The only bounded smooth immersed minimal hypersurface with free boundary on a
catenoid is the flat disk supported at the origin.
\end{uthm}

\begin{uthm}
There does not exist any bounded smooth immersed minimal hypersurface with free
boundary on a cone.
\end{uthm}

This paper is organised as follows.
In Section 2 we study the mean curvature flow and prove our main result that
classifies the asymptotic behaviour of initially graphical rotationally
symmetric data by properties of the support hypersurface $\Sigma$.
When singularities develop, we additionally present some classification of their type.
We apply this in Section 3 to prove classification results for immersed
minimal hypersurfaces with free boundary.

\section{Mean curvature flow with free boundary supported on an oscillating cylinder}

The behaviour of immersions flowing by the mean curvature flow with free
boundary is largely unknown, with available results in the literature
indicating that a complete picture of asymptotic behaviour irrespective of
initial condition is extremely difficult to obtain \cite{stahl1996res,
koeller2007singular}.
Therefore the relevant question is: under which initial conditions is it
possible to obtain a complete picture of asymptotic behaviour?

Working in the class of graphical hypersurfaces is a viable strategy, so long
as the graph condition can be preserved
\cite{thesisvulcanov,wheeler2014meanhyperplane,wheeler2014mean,lambert2012constant}.
In each of these works, global results were enabled by symmetry of the initial
data and/or of the boundary.
Without such symmetries, recent work indicates that graphicality is not in
general preserved \cite{andrewswheeler} (even in the case where
$F_\Sigma(\Sigma)$ is a standard round sphere).

Let us formally set the support hypersurface
$F_\Sigma:\Sigma\rightarrow\R^{n+1}$ to be rotationally symmetric and generated
by the graph of a function $\omega_\Sigma:Oz\rightarrow\R$ over the $Oz$ axis.
We term such a support hypersurface an \emph{oscillating cylinder}.

By convention we let $x=(x_1,\ldots,x_n)$ be a point in $\R^n\subset\R^{n+1}$,
with $n\geq 2$ and denote by $y=|x|$ the length of $x$.
With this convention the profile curve of the support surface lies in a plane generated by $Oy$ and $Oz$ axes.

We write the graph condition on $\omega_\Sigma$ as
\begin{align}
\IP{{\nu}_{\Sigma}(z)}{e_1}\ >\ C_\Sigma \geq 0,\label{Sigma_graph}
\end{align}
where $C_\Sigma$ is a global constant, ${\nu}_{\Sigma}$ the normal to
${\omega}_{\Sigma}$, and $\IP{\cdot}{\cdot}$ is the standard inner
product in ${\R}^{n+1}$.
Our convention is that $\nu_\Sigma$ points away from the interior of the
evolving hypersurface.

Let us now describe how a rotationally symmetric graphical mean curvature flow
with free boundary $F:D^n\times[0,T)\rightarrow\R^{n+1}$ satisfying
\eqref{MCFwFB} can be represented by the evolution of a scalar function (the
graph function).
Let us set $D(t)=(0,r(t))\subset\R$.
The Neumann boundary is at $\partial D(t)=r(t)$.
The left-hand endpoint of $D(t)$, the zero, is not a true boundary point.
It arises from the fact that the scalar generates a radially symmetric graph
that is topologically a disk.
The coordinate system degenerates at the origin and so it is artificially
introduced as a boundary point.
This is however a technicality, and no issues arise in dealing with quantities
at this fake boundary point, since by symmetry and smoothness we have that
the radially symmetric graph is horizontal at the origin.

We represent the mean curvature flow of a radially symmetric graph
$F:D^n\times[0,T)\rightarrow\R^{n+1}$
by the evolution of its graph function $\omega: D(t)\times
[0,T)\rightarrow \R$, that must satisfy the following:
\begin{align}
\frac{\partial \omega}{\partial t}\ \  &=  \frac{d^2\omega}{dy^2}\
\frac{1}{1+(\frac{d\omega}{dy})^2}+\frac{d\omega}{dy}\
\frac{n-1}{y}&&~~\text{ on }~~(0,r(t))\times[0,T),
\label{Neumannproblem}\\
\IP{{\nu}_\omega}{{\nu}_{\Sigma}} &= 0 \text{ and
}r(t)={\omega}_{\Sigma}(\omega(r(t),t))&&~~\text{ on }~~
r(t)\times[0,T),\notag\\
\lim_{y\rightarrow0}\frac{1}{y}&\frac{d\omega}{dy}(y)\text{ exists, and }\notag\\
\omega(y,0) &= \omega_0&&~~\text{ on }~~(0,r(0)).\notag
\end{align}
where $\omega_0:(0,r(0))\rightarrow \R$ generates the initial graph, $\omega_0
\in C^2((0,r(0)))$, that also satisfies the boundary Neumann boundary condition
$\IP{{\nu}_{\omega_0}}{{\nu}_{\Sigma}}= 0$ at $r(0)$.

Note that in this representation the graph direction for $\omega_\Sigma$ is
perpendicular to the graph direction for $\omega$.
(Contrast with \cite{vmwheeler2012rotsym}.)
The two graphs share the same axis of revolution.
Examples of this include graphs evolving inside a vertical catenoid neck or
inside the hole of a vertical unduloid.

\subsection{Existence}

We prove global existence of solutions to \eqref{Neumannproblem} by obtaining
uniform $C^1$ estimates.
The problem \eqref{Neumannproblem} is a quasilinear second-order PDE on a
time-dependent domain with a Neumann boundary condition.
The change in domain can be calculated (see \eqref{rprime}) and depends only on
$\omega_\Sigma$, $\omega'$, and $\omega''$.
The local unique existence of a solution in this setting is standard and has
been discussed in detail in \cite{thesisvulcanov,vmwheeler2012rotsym}.

We note that the uniqueness of a solution shows that the representation
\eqref{Neumannproblem} of a solution to \eqref{MCFwFB} is preserved.

Our first main result is the following.
\begin{thm}[Long time existence]
Let $\omega_\Sigma$ and ${\omega}_0$ be defined as above. Assume
\eqref{Sigma_graph}, and that
\begin{equation}
\label{nopinch}
\text{there is no point $z^*$ where $\omega_{\Sigma}(z^*)=0$}\,.
\end{equation}
We further assume for negative and positive infinity that either one of
\begin{equation}
\label{EQcondn}
\text{the limit $\lim_{z\rightarrow\pm\infty}\omega_\Sigma(z)$ does not exist}
\end{equation}
or
\begin{equation}
\label{noshrinkingatinfinity}
\text{ there exists an $|\alpha| > 0$ such that }
z\frac{d\omega_\Sigma}{dz}(z) > 0\text{ for all }
\begin{cases}
z>\alpha, \text{if }\alpha > 0,
\\
z<\alpha, \text{if }\alpha < 0,
\end{cases}
\end{equation}
hold.
Then there exists a global smooth solution $\omega: D(t)\times [0,T)\rightarrow
\R$ to the
problem \eqref{Neumannproblem} that converges smoothly to
$\omega_\infty:D(\infty)\rightarrow\R$.
The function $\omega_\infty$ is smooth and generates a minimal surface.
\label{LTE}
\end{thm}

\begin{rmk}
The class of support hypersurfaces that satisfy \eqref{noshrinkingatinfinity}
above at both positive and negative infinity are those whose derivative is monotone
outside a compact subset with the correct sign.
Examples of such include the catenoid.

The catenoid also satisfies condition \eqref{EQcondn}.
An example that satisfies \eqref{noshrinkingatinfinity} but not \eqref{EQcondn}
is
\[
\omega_\Sigma(z) = 2 - e^{-z^2}\,.
\]
More generally, any $\omega_\Sigma$ whose derivative is monotone increasing
outside a compact set and converges at infinity satisfies
\eqref{noshrinkingatinfinity} but not \eqref{EQcondn}.

Examples that satisfies \eqref{EQcondn} but not \eqref{noshrinkingatinfinity}
include the unduloids.
Examples that satisfy \eqref{EQcondn} for $z\rightarrow\infty$ and
\eqref{noshrinkingatinfinity} for $z\rightarrow-\infty$ can be pasted together
using those mentioned above; for example, mollify the positive part of a
catenoid with the negative part of $z\mapsto 2-e^{-z^2}$.
\end{rmk}

\begin{rmk}
The condition \eqref{nopinch} prevents $\omega_\Sigma$ from pinching on the
axis of rotation.
If this condition is violated, we expect that some solutions to
\eqref{Neumannproblem} develop finite-time singularities.
This case is treated in detailed in subsection 2.3.
At such points we do not require that $\omega_\Sigma$ is smooth; that is, we
allow cones.
\end{rmk}

\begin{rmk}
The condition \eqref{noshrinkingatinfinity} prevents the solution from
shrinking and sliding off to infinity.
This would happen for a solution supported in the $\Sigma$ generated by
$\omega_\Sigma(z) = 1 + e^{-z^2}$.
Clearly for such solutions we can not expect convergence.
\end{rmk}

For the proof, we use standard machinery of parabolic theory (see for example
\cite{lieberman1996second,ladyshenzkaya1968parabolic}) and its variants for
time-dependent domains as discussed in \cite{vmwheeler2012rotsym,thesisvulcanov}.
In these references a maximum principle is proved; we will apply this without
further reference.

The condition $\IP{{\nu}_\omega}{{\nu}_{\Sigma}} = 0$ can be written in a
simpler way if we take into account the fact that we are working with two
graph functions.
The outer normal to $\omega$ is given by
\[
{\nu}_{\omega}=\frac{1}{\sqrt{1+(\frac{d\omega}{dy})^2}}\big(-\frac{d\omega}{dy},1\big)\,.
\]
For the unit normal to $\omega_\Sigma$ we need to rotate and translate the axes.
We find
\[
{\nu}_{\Sigma}=\frac{1}{\sqrt{1+(\frac{d{\omega}_{\Sigma}}{dz})^2}}\big(1,-\frac{d{\omega}_{\Sigma}}{dz}\big)\,.
\]
This transforms the Neumann boundary condition into
\begin{align}
\frac{d \omega}{dy}(r(t),t)=-\frac{d {\omega}_{\Sigma}}{dz}(\omega(r(t),t))~~
\text{  for all   } t\in[0,T),\label{Neumanncondition}
\end{align}

and gives us the following uniform boundary gradient estimate for $\omega$.

\begin{lem}[Uniform boundary gradient estimates]
Let $\omega_\Sigma$ and ${\omega}_0$ be defined as above.
Assume \eqref{Sigma_graph}.
Then
\begin{align*}
\bigg|\frac{d \omega}{dy}(r(t),t)\bigg|\ \leq \ \sqrt{\frac{1}{C_\Sigma}-1}
\end{align*}
for all $t\in[0,T)$.
\label{gradientestimates}
\end{lem}

\begin{proof}
The Neumann condition \eqref{Neumanncondition} gives a bound on the
gradient of $\omega$ in terms of the gradient $\omega_\Sigma$,
however the constant is not particularly clear.
To find this constant we once more look at the boundary condition. Due to the
rotational symmetry we see that the unit normal of $\Sigma$ is, on the boundary,
the same vector as the tangent vector to the evolving graphs. Thus
\begin{align*}
\nu_{\Sigma}(\omega(r(t),t))\ =\ \frac{1}{\sqrt{1+(\frac{d \omega}{dy}(r(t),t))^2}}\bigg( 1,\ \frac{d \omega}{dy}(r(t),t)\bigg),
\end{align*}
for all $t\in[0,T)$. Replacing this into the graph condition \eqref{Sigma_graph} we find
\begin{align*}
\IP{\frac{1}{\sqrt{1+(\frac{d \omega}{dy}(r(t),t))^2}}\bigg(1,\ \frac{d \omega}{dy}(r(t),t)\bigg)}{e_1} \ \geq\  C_\Sigma.
\end{align*}
Simplifying we obtain
\begin{align*}
\sqrt{1+\bigg(\frac{d \omega}{dy}(r(t))\bigg)^2}\  \leq\ \frac{1}{C_\Sigma},
\end{align*}
which yields the desired estimate.
\end{proof}

\begin{proof}[Proof of Theorem \ref{LTE}]
As we allow the boundary to possibly oscillate, height bounds are not immediate.
The maximum principle applies to $|\omega|$, yielding that $|\omega|$ is
bounded by the maximum of its boundary and initial values.
The main task is to control the value of $|\omega|$ on the Neumann boundary.

For the Hopf lemma to work in excluding new maxima on the Neumann boundary we
need to have a certain sign on the directional derivative $\frac{d
\omega}{dy}(r(t),t)$. This is not possible since this quantity changes sign with
the gradient of the $\Sigma$.
This is evident from \eqref{Neumanncondition}.

To obtain height bounds we proceed as follows. First suppose that there exist
two points $z_{sup}$ and $z_{inf}$ such that we have
\[
z_{inf}\ <\ \min_{[0,r(0)]} \omega_0<\ \max_{[0,r(0)]} \omega_0\ < z_{sup}
\]
and
\[
\frac{d\omega_{\Sigma}}{dz}(z_{sup})
 =\frac{d\omega_{\Sigma}}{dz}(z_{inf})
 = 0
\,.
\]
Then the initial graph is contained in
\[
\Omega = \{(x,z)\,:\,z\in(z_{inf},z_{sup}),\quad |x| < \omega_\Sigma(z)\}\,.
\]
The set $\Omega$ is bounded by the support hypersurface and a minimal
disk at each end.
These act as barriers for the flow: by the avoidance principle we find
\[
z_{inf} < \omega < z_{sup}\,.
\]
If such points $z_{inf}$ and $z_{sup}$ do not exist, then there do not exist
flat disks supported on $\Sigma$ disjoint from the initial graph $\omega_0$
that can be used as barriers.

Assume that there is no such disk in
\[
U = U^+ \cup U^-
\]
where
\[
U^+ = \{(x,z)\,:\,z \ge 0,\, z>\max_{[0,r(0)]} \omega_0\text{ or } z<\min_{[0,r(0)]}\omega_0\}
\]
and
\[
U^- = \{(x,z)\,:\,z < 0,\, z>\max_{[0,r(0)]} \omega_0\text{ or } z<\min_{[0,r(0)]}\omega_0\}.
\]
Each of $U^+$ and $U^-$ have at most two components, one finite and bounded by the plane $z=0$ and another unbounded.
Let $z$ be in the unbounded component of $U^+$.
There are two cases.

{\bf Case 1.} Condition \eqref{noshrinkingatinfinity}
is satisfied on an unbounded component $U^{++}$ of $U^+$.
On this component, the derivative $\frac{d\omega_{\Sigma}}{dz}$ has a sign.
As we know that for sufficiently large $z$, the derivative $\frac{d\omega_{\Sigma}}{dz}(z)$ is positive, in this case it
must be positive on all of $U^{++}$.
Now the boundary condition \eqref{Neumanncondition} implies that
\[
\frac{d\omega}{dy}(z)<0\,.
\]
The Hopf lemma implies that $\omega$ may never reach such a region.

Similarly, if condition \eqref{noshrinkingatinfinity} is satisfied on an
unbounded component $U^{--}$ of $U^-$, and $z$ is in the unbounded component of
$U^-$, then on this component $\frac{d\omega_{\Sigma}}{dz}$ has a
sign, and as we know that for sufficiently large $z$ the derivative
$\frac{d\omega_{\Sigma}}{dz}(z)$ is negative, in this case it must be negative
on all of $U^{--}$.
Now the boundary condition \eqref{Neumanncondition} implies that for all such $z$
\[
\frac{d\omega}{dy}(z)>0\,.
\]
The Hopf lemma again implies that $\omega$ may never reach such a region.

{\bf Case 2.} Condition \eqref{EQcondn} is satisfied on an unbounded component
$U^{++}$ of $U^+$.
As no minimal disk exists on this component, the derivative $\frac{d\omega_{\Sigma}}{dz}$ again has a sign.
If the sign is positive, then the Hopf lemma applies as in Case 1 above.
If the sign is negative, then as $z\rightarrow\infty$, the function $\omega_\Sigma$ is uniformly bounded from below (by the no pinching condition \eqref{nopinch}) and decreasing.
Therefore it converges, violating \eqref{EQcondn}.

If condition \eqref{EQcondn} is satisfied on an unbounded component $U^{--}$ of $U^{-}$ then the derivative is negative.
If it were positive, then similarly as above this is in contradiction with \eqref{EQcondn}.

Therefore in either case the evolving surfaces are contained within a compact region of $\R^{n+1}$, and so the graph
function $\omega$ is uniformly bounded.

Thus we are left with obtaining gradient estimates for the evolving graphs $\omega:D(t)\times[0,T)\rightarrow\R$.

Let us set $M_t := F(D^n,t)$.
Following \cite{ecker1989mce} we consider the quantity $v=\IP{{\nu}_{M_t}}{e_{n+1}}^{-1}$, which is modulo a tangential
diffeomorphism equal to $\sqrt{1+(\frac{d\omega}{dy})^2}$.
The function $v$ satisfies
\[
\Big(\frac{d}{dt} - \Delta_{M_t}\Big) v \leq 0
\]
and this allows us to apply the maximum principle.
Since the problem deals with evolving hypersurfaces with boundary we have that the maximum of the gradient is controlled
by the maximum between the initial values and the boundary values.
In the graphical setting, this translates to the following estimate:
\begin{align*}
\sup_{(0,r(t))} \bigg|\frac{d \omega}{dy}\bigg|
\le \max \bigg\{\sup_{(0,r(0))}\bigg|\frac{d \omega}{dy}\bigg| ,\
\sup_{s\in[0,t]}\bigg|\frac{d \omega_0}{dy}(r(s),s)\bigg| ,\ \ \sup_{s\in[0,t]}\bigg|\frac{d \omega}{dy}(0,s)\bigg|
\bigg\}\,,
\end{align*}
for all $t\in[0,T]$.

We now refine this by considering maxima at the boundary.
At the artificial boundary point ($y=0$) the gradient function vanishes due to rotational symmetry; that is,
\[
\frac{d \omega}{dy}(0,t)=0\,.
\]
Lemma \ref{gradientestimates} gives a uniform estimate for the gradient on the Neumann boundary.
We can therefore conclude that
\begin{align*}
\sup_{(0,r(t))} \bigg|\frac{d \omega}{dy}\bigg|
\le \max \bigg\{ \sup_{(0,r(0))}\bigg|\frac{d \omega_0}{dy}\bigg|,\  \sqrt{\frac{1}{C_\Sigma}-1}\bigg\}\,.
\end{align*}
Having obtained a-priori uniform $C^1$ estimates for $\omega:D(t)\times[0,T)\rightarrow\R$, the quasilinear parabolic
operator \eqref{Neumannproblem} may be considered to be linear with bounded coefficients, and for such a problem global
existence is standard.
Convergence to minimal hypersurfaces is guaranteed by bounded initial area:
We calculate
\[
\frac{d}{dt}\int_{D^n}\,d\mu = -\int_{D^n}|H|^2d\mu
\]
which implies
\[
\int_0^\infty\int_{D^n}H^2d\mu\,dt \le \int_{D^n}\,d\mu\bigg|_{t=0} = c\,.
\]
Finally, we have $|D(t)| \ge \inf \omega_\Sigma > 0$ by assumption, so that $\omega$ does not vanish (c.f. Theorem
\ref{thmsingularities}).
Since all derivatives are uniformly bounded, we may apply a compactness theorem to conclude that $M_t\rightarrow
M_\infty$ and that the mean curvature of $M_\infty$ is identically zero. This argument has been used before by many
authors, see for example \cite{huisken1989npm,buckland2005mcf,thesisvulcanov}.
\end{proof}

\begin{rmk}
On the free Neumann boundary, the rotational symmetry of the solution prevents tilt behaviour.
This occurs when the normal to the graph becomes parallel to the vector field of rotation for $\Sigma$.
This behaviour is explained in much greater detail in \cite{thesisvulcanov} and it is present in many situation of free
boundary problems \cite{andrewswheeler}, thus the need to use the rotationally symmetry in constructing the barriers
needed to show the elliptic results.
\end{rmk}

\subsection{Convergence}

After showing that the solution to the problem \eqref{Neumannproblem} exists for all times we are interested in studying
the precise shape that it attains in the limit as $t\rightarrow \infty$, knowing already that it is a minimal
hypersurface.
In fact, the theory of minimal hypersurfaces (note that the boundary of this disk is a circle) implies that the limit is
a flat disk.
However, we may prove this directly without requiring the general theory, and so we contribute a proof here.
We also give some related results of interest.

\begin{thm}[Convergence to flat disks]
Under the hypotheses of Theorem \ref{LTE}, the global smooth solution
$\omega:D(t)\times[0,\infty)\rightarrow\R$ to the problem \eqref{Neumannproblem}
satisfyies
\begin{align*}
\displaystyle \lim_{t\rightarrow \infty}\sup_ {[0,r(t)]}\bigg|\frac{d \omega}{dy}(r(t),t)\bigg|=0\,,
\end{align*}
that is, the
solution converges to a flat disk as $t\rightarrow \infty$.
\label{thmconvergence}
\end{thm}
\begin{proof}
First we prove that the gradient on the boundary vanishes.

Let us denote by $u(x)=\omega_\infty(|x|)$, $u:D_\infty\rightarrow\R$, $D_\infty=\{x\in\R^n\ :\ y=|x|\in
Dom(\omega_\infty) \}$.
The mean curvature of $u$ is
\[
H=-div\bigg(\frac{\D u}{\sqrt{1+|\D u|^2}}\bigg)
\]
where $\D$ and $div$ are the gradient and divergence in $\R^n$ respectively.
We can then compute using divergence theorem and denoting by $\nu_{\partial
D_\infty}$ the outer pointing normal to the boundary of the domain $D_\infty$:
\begin{align*}
0\ &=\ -\int_{D_\infty}H\, dx=\int_{D_\infty} div\bigg(\frac{\D u}{\sqrt{1+|\D u|^2}}\bigg) \, dx\\
&\ =\ \int_{\partial D_\infty} \frac{\D u}{\sqrt{1+|\D u|^2}} \cdot \nu_{\partial
D_\infty} \, dSx\\
&\ =\ \int_{\partial D_\infty} \frac{\frac{d \omega_\infty}{dy}}{\sqrt{1+|\frac{d \omega_\infty}{dy}|^2}} \, dx\\
&\ =\ 2\pi r(\infty)\frac{\frac{d \omega_\infty}{dy}}{\sqrt{1+|\frac{d \omega_\infty}{dy}|^2}}(r(\infty))
\end{align*}
where smoothness of the solution at the rotation axis (i.e. $\frac{d\omega_\infty}{dy}(0)=0$) ensures that the second
boundary term vanishes. This implies that $\frac{d\omega_\infty}{dy}(r(\infty))\equiv0$.

Using this we can show that the gradient of $\omega_\Sigma$ vanishes everywhere:
\begin{align*}
0\ &=\ -\int_{D_\infty}Hu\, dx=\int_{D_\infty} div\bigg(\frac{\D u}{\sqrt{1+|\D u|^2}}\bigg) u\, dx\\
&\ =\ -\int_{D_\infty} \frac{|\D u|^2}{\sqrt{1+|\D u|^2}}\,dx+ \int_{\partial D_\infty} \frac{\D u}{\sqrt{1+|\D u|^2}} \cdot \nu_{\partial
D_\infty} u\, dSx\\
&\ =\ -\int_{D_\infty} \frac{|\D u|^2}{\sqrt{1+|\D u|^2}}\,dx
\end{align*}
where we have used $\frac{d\omega_\infty}{dy}(r(\infty))=0$ for the Neumann boundary and also the smoothness of the solution at the
rotation axis, $\frac{d\omega_\infty}{dy}(0)=0$ to make the boundary term vanish. This implies that $\D u\equiv 0$ and
thus $\frac{d\omega_\infty}{dy}\equiv0$, that is, $\omega_\infty$ is a constant.
\end{proof}

The above calculation implies the following result, which is interesting in its own right.

\begin{lem}
Suppose $F:D^n\rightarrow\R^{n+1}$ is an embedded minimal disk with boundary on an oscillating cylinder $\Sigma$
with axis of revolution $Oz$.
If
\begin{itemize}
\item $F(\partial D)$ is a circle in a plane orthogonal to $Oz$; and
\item $F$ is graphical over a disk orthogonal to $Oz$ (but not necessarily rotationally symmetric),
\end{itemize}
then $F(D^n)$ is a standard flat disk.
\label{LMrig}
\end{lem}

This implies that on the boundary, the gradient of limiting hypersurface will vanish \emph{independent of the angle
imposed by the flow problem}.
This explains the non-compactness of the flow (and consequent appearance of
translators) in cases where the support hypersurface doesn't allow this to
happen. (See \cite{altschuler1994,leithesis,guan1996mean} for further results on
flows with various contact angles.)

\begin{cor}
Suppose $F:D^n\times[0,T)\rightarrow\R^{n+1}$ is as in Lemma \ref{LMrig},
except that at the free boundary where the prescribed angle is $\alpha$, that
is,
\[
\IP{\nu_\omega}{\nu_\Sigma} = \cos\alpha\,.
\]
If there exists no point $z\in Oz$ such that
\[
\frac{\frac{d\omega_\Sigma}{dz}}
{\sqrt{1+\left(\frac{d\omega_\Sigma}{dz}\right)^2}}
= -\cos\alpha
\]
then the flow never reaches an equilibrium.
\end{cor}

\begin{figure}
 \centering
 \begin{center}
  \includegraphics[trim=2cm 13cm 0.5cm 2cm,clip=true,width=0.59\textwidth]{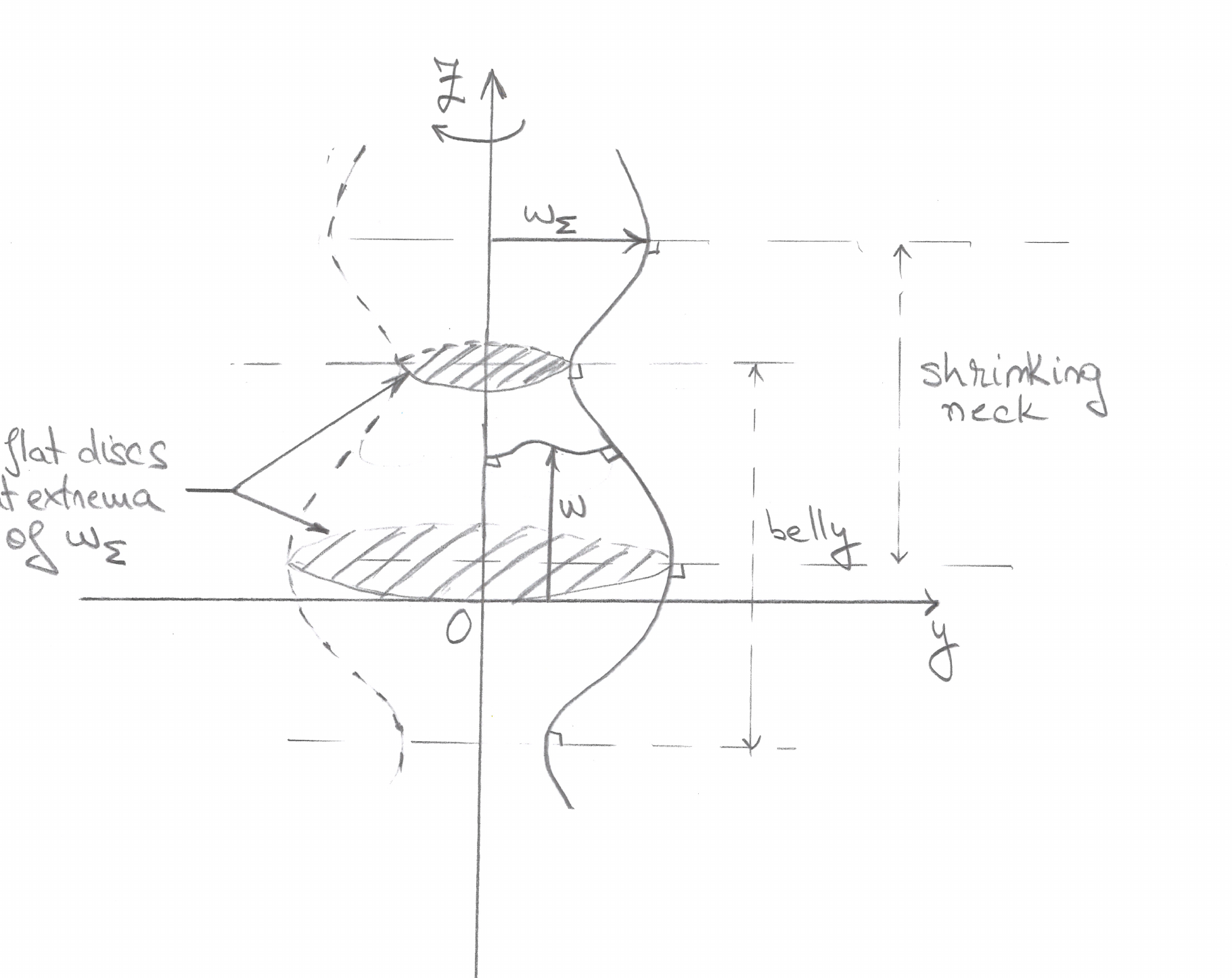}\\
  \captionstyle{\centering}
  \caption{Flat discs at extrema of $\Sigma$.}
  \label{fig:twomin}
  \end{center}
\end{figure}

The graphs generating the contact hypersurface $\Sigma$ are in general oscillating, forming local minima and maxima as
they stretch in both directions of the axis $Oz$. The flow may in general converge to any flat disk supported at a
critical point of $\omega_\Sigma$.
Note that the case of $\Sigma$ being a cylinder has been treated previously in \cite{huisken1989npm}.

Generically, one expects the flow to converge to minimal disks with the smallest possible area.
Such disks are found at local minima of $\omega_\Sigma$; however, it appears difficult to rule out convergence to other
minimal disks for arbitrary data.
In the following we give sufficient conditions on the initial data (that $\Omega$ is contained in a shrinking neck
region, see Definition \ref{DFbellies}) to guarantee convergence to a minimal disk supported on a local minimum of
$\omega_\Sigma$.

\begin{defn}[Bellies and necks, see Figure \ref{fig:twomin}]
\label{DFbellies}
Let $\Sigma$ be an oscillating cylinder.
A \emph{region} of $\Sigma$ is any set ($z_1, z_2 \in [-\infty,\infty]$)
\[
\Theta(z_1,z_2) = \{(x,z)\,:\,|x|<\omega_\Sigma(z),\, z\in(z_1,z_2)\}\,.
\]
A \emph{shrinking neck region} of $\Sigma$ is any region $\Theta(z_1,z_2)$ where for all $z\in(z_1,z_2)$
\[
\frac{d\omega_\Sigma}{dz}(z) = 0\quad\Longrightarrow\quad
\text{$\omega_\Sigma(z)$ is a weak local minimum value for $\omega_\Sigma$}\,.
\]

A \emph{belly region} of $\Sigma$ is any region $\Theta(z_1,z_2)$ where for all $z\in(z_1,z_2)$
\[
\frac{d\omega_\Sigma}{dz}(z) = 0\quad\Longrightarrow\quad
\text{$\omega_\Sigma(z)$ is a weak local maximum value for $\omega_\Sigma$}\,.
\]
\end{defn}

\begin{thm}[Convergence in shrinking necks]
Assume the hypotheses of Theorem \ref{LTE}.
Suppose that the initial data $F_0(D^n)$ is contained in a shrinking neck region $\Theta(z_1,z_2)$.
Then the global solution $F:D^n\times[0,\infty)\rightarrow\R^{n+1}$ to \eqref{MCFwFB} converges to a flat disk supported
at a local minimum of $\omega_\Sigma$ in $\Theta(z_1,z_2)$.
If there is just one such minimum at $z^*\in(z_1,z_2)$ then the flat disk supported at $\omega_\Sigma(z^*)$ is the
unique limit of all solutions to \eqref{Neumannproblem} with initial data in $\Theta(z_1,z_2)$.
\label{minimumflatdiscs}
\end{thm}

\begin{proof}
Theorems \ref{LTE} and \ref{thmconvergence} yield that the solution exists for all time and converges to a flat disk.
It remains to identify which disks may serve as limits for the flow.

All minimal disks in a shrinking neck region are located at local minima of $\omega_\Sigma$ by definition.
Therefore we will be finished if we can prove that there exists a shrinking neck region $\Theta(w_1,w_2)$ such that for
all $(x,t) \in D^n\times[0,\infty)$,
\begin{equation}
\label{Claim1}
F(D^n,t) \subset \Theta(w_1,w_2)\,.
\end{equation}
We extend the given shrinking neck region in either direction until we reach a critical point of $\omega_\Sigma$.
More precisely, let us take
\[
w_1 = \sup\bigg( \bigg\{z\in(-\infty,z_1]\,:\, \frac{d\omega_\Sigma}{dz}(z) = 0 \bigg\} \cup \big\{ -\infty \big\} \bigg)
\]
and
\[
w_2 = \inf\bigg( \bigg\{z\in[z_2,\infty)\,:\, \frac{d\omega_\Sigma}{dz}(z) = 0 \bigg\} \cup \big\{ \infty \big\} \bigg)
\]
Clearly $\Theta(z_1,z_2)\subset \Theta(w_1,w_2)$, and $\Theta(w_1,w_2)$ is a region of $\Sigma$.
To see that $\Theta(w_1,w_2)$ is a shrinking neck region, let $z_0\in\Theta(w_1,w_2)\setminus\Theta(z_1,z_2)$ be a point
where
\[
\frac{d\omega_\Sigma}{dz}(z_0) = 0\text{ and }
\text{$\omega_\Sigma(z_0)$ is a local maximum value for $\omega_\Sigma$}\,.
\]
If $\omega_\Sigma'(z_1) = 0$ then $w_1 = z_1$; similarly for $z_2$.
If both conditions are satisfied, then the set $\Theta(w_1,w_2)\setminus\Theta(z_1,z_2)$ is empty, and such a $z_0$ can not exist.

Suppose otherwise. Then either $z_0\in(w_1,z_1)$ or $z_0\in(z_2,w_2)$.
Suppose the former.
Since $\omega_\Sigma'(z_0) = 0$, $z_0 \in \{z\in(-\infty,z_1]\,:\, \omega_\Sigma'(z) = 0\}$ and $z_0 > w_1$.
This contradicts the definition of $w_1$.
Similarly, $z_0\in(z_2,w_2)$ contradicts the definition of $w_2$.
Therefore such a $z_0$ can not exist.

We now claim \eqref{Claim1}.
Let us prove this by contradiction.
Suppose there exists a sequence of points in space-time $((x_n,z_n),t_n)\in F(D^n,t_n)\times[0,\infty)$ such that
$z_n \rightarrow z_\infty$ such that $z_\infty \le w_1$ or $z_\infty \ge w_2$.
Let us first bring $z_\infty\le w_1$ to a contradiction.

If $z_\infty = -\infty$, this contradicts the height bound from Theorem
\ref{LTE}.
Therefore the only way for $z_n\rightarrow z_\infty$ such that $z_\infty \le w_1$ is if $w_1$ is finite.
Then by definition of $w_1$ we have $\frac{d\omega_\Sigma}{dz}(w_1) = 0$ and so there exists a
minimal disk supported at $w_1$ that serves as a barrier for the solution.
Therefore we are left with the case where $z_\infty = w_1$.
In this case, we have for sufficiently large $n$
\[
\frac{d\omega_\Sigma}{dz}(z_n) < 0\,.
\]
(The strict sign follows from the definition of $w_1$ and the use of the flat disk as a barrier.)
The boundary condition then yields $\frac{d\omega}{dz}(r(t_n),t_n) > 0$.
Therefore there is no maximum for $\omega(\cdot,t_n)$ at it's Neumann boundary.
However, by assumption, the graphs $\omega(\cdot,t_n)$ are moving downward to the flat disk.
Therefore there must be a new minimum for $\omega(\cdot,t_n)$, or equivalently, a new maximum for
$|\omega(\cdot,t_n)|^2$.
This maximum must be either at the axis of rotation ($y = 0$) or in $(0,r(t_n))$.
The parabolic evolution equation for $\omega$ implies that
\begin{align*}
\sup_{(0,r(t))} \big|\omega \big|^2
\le \max \bigg\{\sup_{(0,r(0))}\big|\omega\big|^2 ,\
\sup_{s\in[0,t]}\big|\omega(r(s),s)\big|^2 ,\ \ \sup_{s\in[0,t]}\big|\omega(0,s)\big|^2 \bigg\}\,.
\end{align*}
We already ruled out new maxima on the Neumann boundary. The Hopf Lemma implies that new maxima are also impossible at
the axis of rotation, since there $\frac{d\omega}{dz} = 0$ by symmetry. The only case remaining is that the new maxima
occur on the interior, which is clearly a contradiction.

Therefore $z_\infty > w_1$. A similar argument shows that $z_\infty < w_2$, and so the claim \eqref{Claim1} is proved.

\end{proof}

\begin{rmk}[$\Sigma$ catenoid]
If $\Sigma$ is a catenoid or it has only one global minimum then $\Theta(-\infty,\infty)$ is a shrinking neck region.
Theorem \ref{minimumflatdiscs} above then yields that solutions converge as $t\rightarrow \infty$ to the unique flat
disk perpendicular to $\Sigma$ at this point (c.f. the analogous result in \cite{vmwheeler2012rotsym}).
\end{rmk}

If the initial data is contained in a maximal finite belly region, then it is trapped in this region, by comparison with
flat disks at either end (c.f. the proof of Theorem \ref{minimumflatdiscs} above).
If the initial data is to one side of the highest (or lowest) flat disk in the belly region, then it is also in a
shrinking neck region, and the previous theorem applies.
If the initial data \emph{intersects any flat disk} in the belly region, then the asymptotic behaviour of the flow
becomes more complicated.
In the following result we give a sufficient conditions that guarantees the flow (even if initially in a belly region
intersecting a flat disk) moves out of the belly region and converges to a flat disk in a shrinking neck region.

\begin{figure}
 \centering
 \begin{center}
  \includegraphics[trim=3cm 13cm 0.5cm 3cm,clip=true,width=0.69\textwidth]{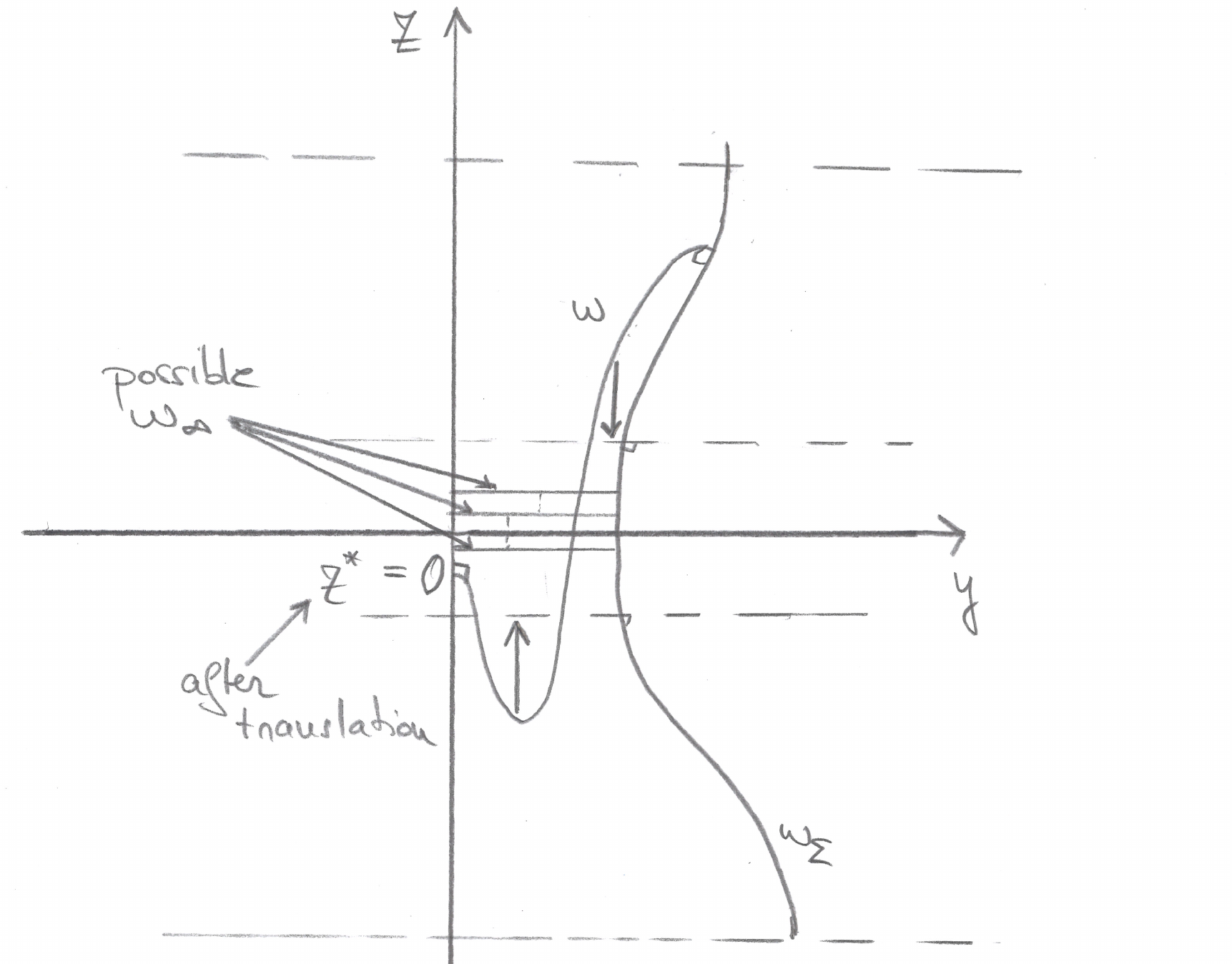}\\
  \captionstyle{\centering}
  \caption{Uniqueness of limiting disks for flows with initial data contained in a shrinking neck region.}
  \label{fig:neck}
  \end{center}
\end{figure}

\begin{thm}[Convergence for initial data in bellies]
Assume the hypotheses of Theorem \ref{LTE}.
Suppose that the initial data $F_0(D^n)$ is contained in a belly region $\Theta(z_1,z_2)$ such that it intersects a flat
disk in $\Theta(z_1,z_2)$ and:
\begin{enumerate}
\item[(a)] $H(\cdot,0)<0$ and $\omega_0(r(0)) > \alpha$ where $\alpha$ is the $z$-coordinate of the highest minimal disk in $\Theta(z_1,z_2)$;
or
\item[(b)] $H(\cdot,0)>0$ and $\omega_0(r(0)) < \beta$ where $\beta$ is the $z$-coordinate of the lowest minimal disk in
$\Theta(z_1,z_2)$.
\end{enumerate}
Then the global solution $F:D^n\times[0,\infty)\rightarrow\R^{n+1}$ to \eqref{MCFwFB} converges to a flat disk supported
at a local minimum of $\omega_\Sigma$ in a shrinking neck region.
If $\Theta(z_1,z_2) \subset \Theta(z_{min},z_{max})$ where $\Theta(z_{min},z_{max})$ is a maximal belly region and $-\infty <
z_{min} < z_{max} < \infty$ then
the flat disk supported at $\omega_\Sigma(z_{max})$ is the
unique limit of all solutions to \eqref{Neumannproblem} with initial data satisfying (a) and
the flat disk supported at $\omega_\Sigma(z_{min})$ is the
unique limit of all solutions to \eqref{Neumannproblem} with initial data satisfying (b).
\label{minimumflatdiscsH}
\end{thm}
\begin{rmk}
A sign on the mean curvature does not imply that the profile $\omega$ is convex or concave.
Note that if $\alpha = \beta$ then we set $z^* = \alpha = \beta$ (see Figure \ref{fig:belly}).

If the belly region is infinite on either side, then we do not expect solutions to converge.
This is not possible here, as it would contradict the assumptions of Theorem \ref{LTE}.
\end{rmk}

Before we start the proof of the theorem we require a result on preservation of the sign of the mean curvature for mean
curvature flow with free boundary.
This is due to Stahl \cite{thesisstahl}.

\begin{figure}
 \centering
 \begin{center}
  \includegraphics[trim=3cm 14cm 0.5cm 3cm,clip=true,width=0.59\textwidth]{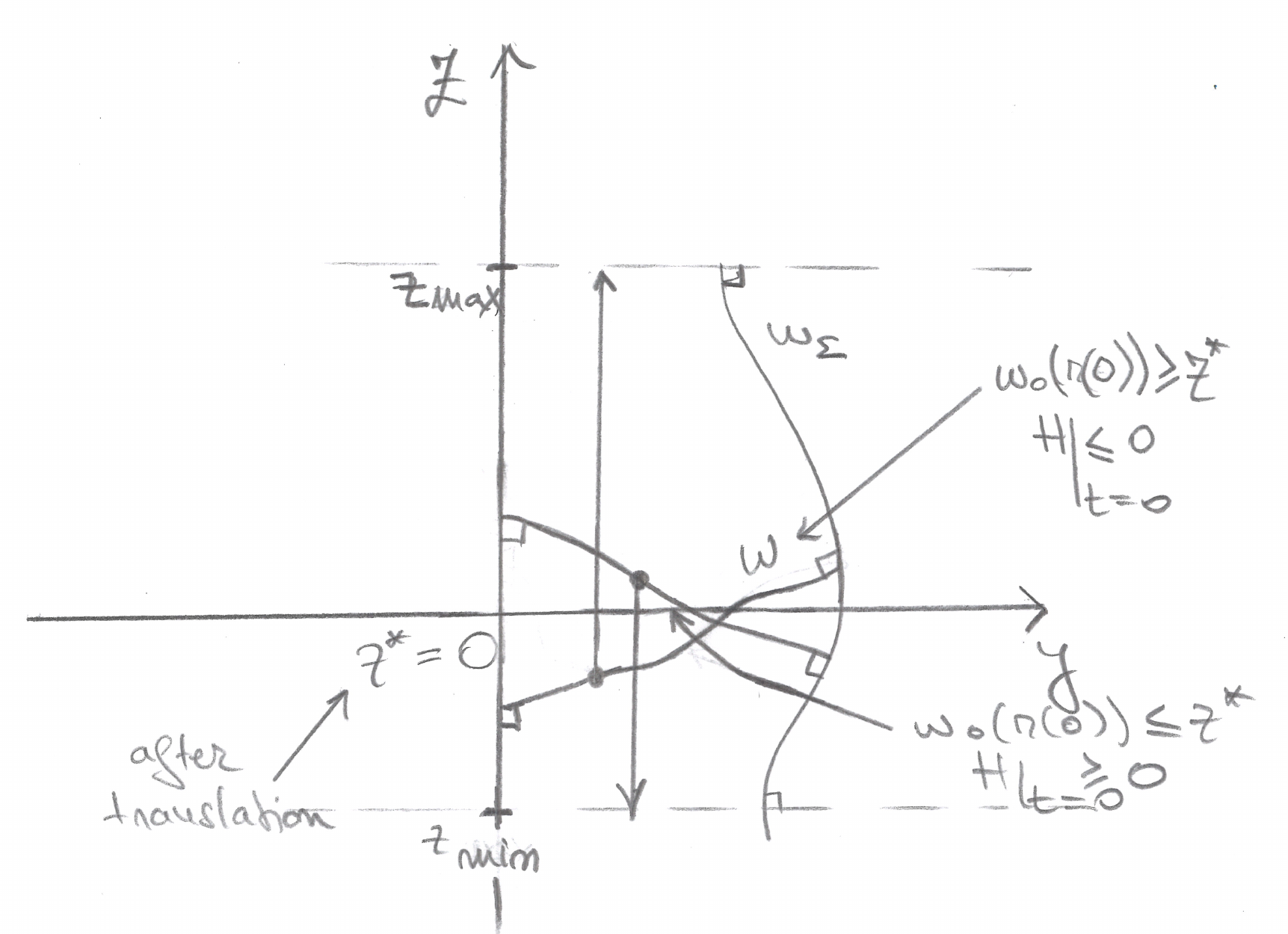}\\
  \captionstyle{\centering}
  \caption{Uniqueness of flat disks in belly regions}
  \label{fig:belly}
  \end{center}
\end{figure}

\begin{prop}[\cite{thesisstahl}]
Let $H(\cdot,0)\geq (\leq)\, 0$ everywhere on $M_0$.
Then $H(\cdot,t)\geq (\leq)\, 0$ for all $t\geq 0$ where $M_t$ a solution
of the mean curvature flow with free boundary.
\end{prop}

For completeness we sketch the proof.
It is based on the use of the maximum principle and the fact that on the boundary the directional derivative of the mean
curvature is equal to the mean curvature multiplied by a component of the second fundamental form of $\Sigma$ at that
point.
The Hopf Lemma then yields a contradiction, for any smooth $\Sigma$, with the appearance of a new zero (maximum or
minimum) on the boundary.
The strong maximum principle yields the strict sign for all strictly positive times.

\begin{proof}[Proof of Theorem \ref{minimumflatdiscsH}]
Theorems \ref{LTE} and \ref{thmconvergence} yield that the solution exists for all time and converges to a flat disk.
Suppose we are in situation (a).
First we translate the $Oy$ axis so that $\alpha = 0$, and $z_1 \le 0 \le z_2$.
After this translation, the definition of belly region implies
\begin{align*}
z\frac{d \omega_\Sigma}{dz}\ \leq\ 0,
\end{align*}
for all $z\in\Theta(z_1,z_2)$.
As in the proof of Theorem \ref{minimumflatdiscs}, consider the maximal belly region $\Theta(z_{min},z_{max}) \supset
\Theta(z_1,z_2)$.
By assumption, neither of $z_1$, $z_2$ may be infinite, and so by definition of $z_{min}$, $z_{max}$, there exist flat
disks on the boundary of $\Theta(z_{min},z_{max})$ supported on $\Sigma$.

We are in the case of negative mean curvature.
To show that the graphs will ascend and converge as $t \rightarrow\infty$ to the flat disk at $z=z_{max}$ we look at the
derivative of the boundary point $r(t)$.
Since
\begin{align*}
r(t)\ =\ \omega_\Sigma(\omega(r(t),t))\,,
\end{align*}
we calculate
\begin{align}
\label{rprime}
r'(t) = \frac{d\omega_\Sigma}{dz}\bigg(\frac{\partial\omega}{\partial t} + \frac{d\omega}{dy} r'(t)\bigg)\,.
\end{align}
Substituting in the boundary condition \eqref{Neumanncondition} and $\frac{\partial\omega}{\partial t}= -Hv$
yields
\begin{align*}
r'(t)\ =\ -\frac{H}{v}\frac{d\omega_\Sigma}{dz},
\end{align*}
where we have once again denoted $v=\sqrt{1+(\frac{d\omega}{dy})^2}=\sqrt{1+(\frac{d\omega_\Sigma}{dz})^2} > 0$ at the
boundary points.
Now for all $z\ge0$ (recall the translation)
\begin{align*}
\frac{d \omega_\Sigma}{dz} \le 0.
\end{align*}
This gives us that $r'(t) < 0$ which means that $r(t)$ is decreasing.
As we are in a belly region above the highest flat disk, this implies that $\omega(r(t),t)$ is monotone increasing.
Given that the graphs exist for all times and converge to a flat disk, the first such encountered by the solution is
the flat disk at $z_{max}$.
Since this disk also serves as a barrier for the solution, the proof is finished.
\end{proof}

\begin{rmk}[Height restrictions on the boundary point]
If $\alpha=\beta=z^*$, the restriction on $\omega_0(r(0)) \ge z^*$ is necessary.

This is because otherwise the mean curvature of $\omega_0$ can not be everywhere negative.
If $\omega_0(r(0)) < z^*$, then $\omega_0$ would have to turn after passing the translated $Oy$ axis so that it reaches
the axis of rotation orthogonally, creating a mean convex region.
To see this note that by \eqref{Neumanncondition}, for all $z < z^*$ we have
$\frac{d\omega_0}{dy}(r(0))\ <\ 0$.
At the rotation axis the gradient is vanishing by smoothness, that is, $\frac{d\omega_0}{dy}(0)=0$.
This implies that there exits a point $y^*\in (0,r(0)]$ such that $\frac{d^2\omega_0}{dy^2}(y^*)< 0$.
Otherwise the gradient would just increase, giving a contradiction.
At this point we calculate the mean curvature:
\begin{align*}
0 > H = -\frac{\frac{d^2\omega_0}{dy^2}}{\sqrt{1+(\frac{d\omega_0}{dy})^2}^3}-\frac{n-1}{y{\sqrt{1+(\frac{d\omega_0}{dy})^2}}} \frac{d\omega_0}{dy}
\end{align*}
and obtain that $\frac{d\omega_0}{dy}(y^*) > 0$.
Since $\frac{d\omega_0}{dy}(r(0)) < 0$, we see that there exists a point $y_2\in (0,r(0)]$ such that
$\frac{d\omega_0}{dy}(y_2)=0$.
Repeating the above by replacing the point at zero with $y_2$, we find a second point $y_3\in (y_2,r(0)]$ with the
property that the gradient vanishes at $y_2$.
Denote by $y_1=0$.
In this way we obtain a sequence of points converging $y_k\rightarrow y_\infty$ as $k\rightarrow \infty$,
such that $\frac{d\omega_0}{dy}(y_k)=0$.
If $y_\infty = r(0)$, then by smoothness of $\omega_0$ we obtain a contradiction with the strict sign by the Neumann condition \eqref{Neumanncondition}.
If $y_\infty \in (0,r(0))$ then by smoothness of $\omega_0$, in a left-neighbourhood of $y_\infty$ we have that
$\omega_0$ is flat. This implies in particular that the first two derivatives of $\omega_0$ vanish there, and so the
mean curvature vanishes also. This is in contradiction with $H < 0$.
\end{rmk}

\begin{rmk}[Initial boundary point at $z^*$]
If the initial data has Neumann boundary tangential to the flat disk at $z^*$, and is disjoint from the flat disk in the
interior, it will immediately move into a shrinking neck region and Theorem \ref{minimumflatdiscs} applies. If it is not
immediately disjoint from the flat disk, then it is either tangential or crosses the flat disk. If tangential, then the
mean curvature is zero at some interior points, and this is a contradiction with the mean curvature having a definite
sign. If it crosses the disk, then the same proof in the above remark applies to show that the mean curvature must
change sign.
\end{rmk}

\begin{figure}
 \centering
 \begin{center}
  \includegraphics[trim=1cm 12cm 1cm 0.5cm,clip=true,width=0.59\textwidth]{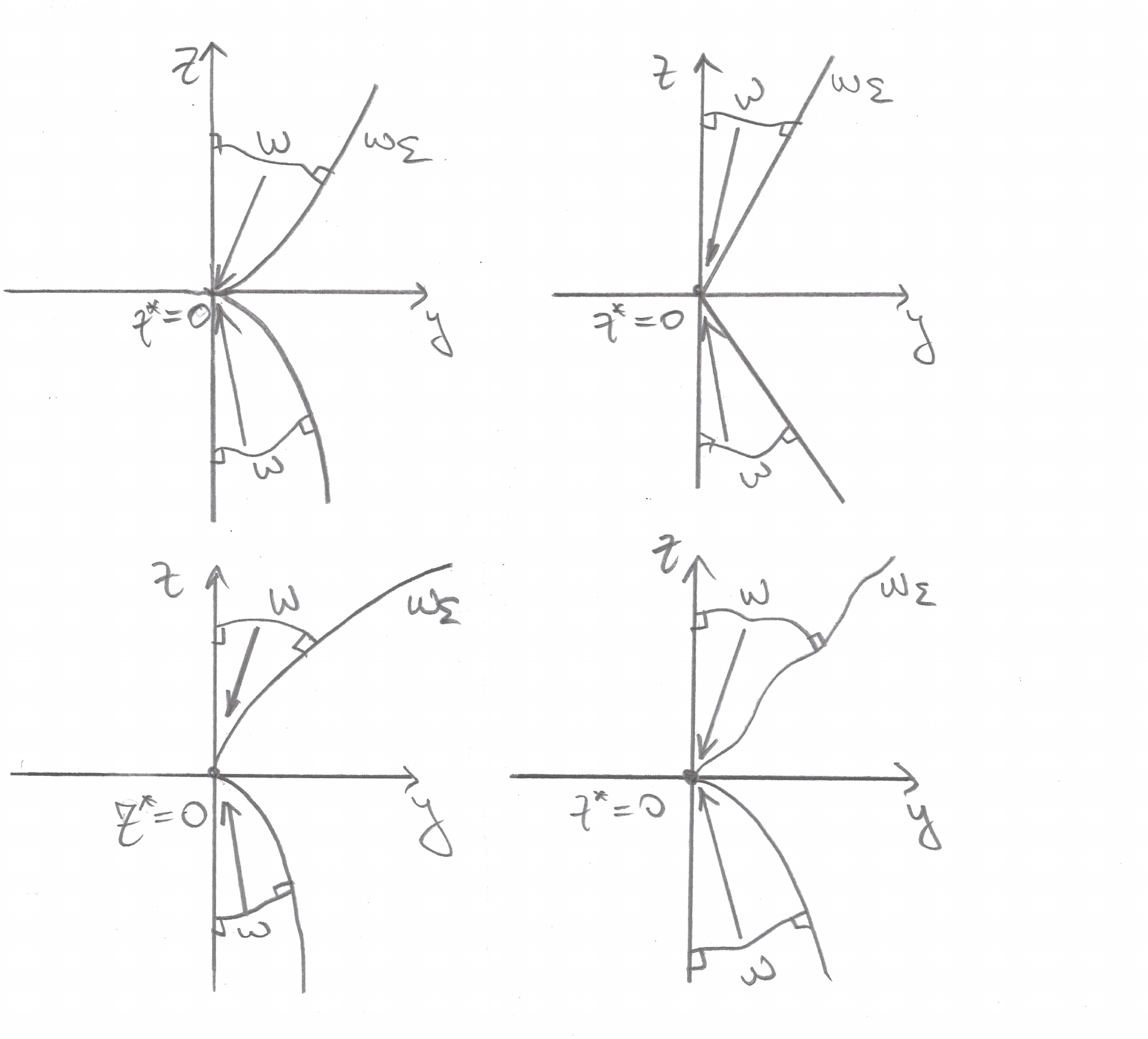}\\
  \captionstyle{\centering}
  \caption{Examples of finite-time singularities.}
  \label{fig:singularity}
  \end{center}
\end{figure}

\subsection{Singularities}

In this section we treat the case when the support hypersurface $\Sigma$ pinches on its axis of rotation; that is, there
exists one or more points $z^*$ such that $\omega_\Sigma(z^*) = 0$.
We do not require that $\Sigma$ is smooth at those points so examples of such support hypersurfaces include cones,
parabolae or hypersurfaces that form cusps at the rotation axis.

\begin{defn}
Let $\omega_\Sigma:Oz\rightarrow[0,\infty)$ be a continuous function.
Assume that $\omega_\Sigma$ is smooth outside finitely many points $P =
\{w_1,\ldots,w_{n_p}\}$, where $\omega_\Sigma(w_i) = 0$; that is,
$\omega_\Sigma\in C^\infty_{loc}(Oz\setminus P)$.
Assume that there exists a compact set $K \supset P$ such that
\begin{equation*}
z\frac{d\omega_\Sigma}{dz}(z) > 0\quad\text{ for all }z\in\R\setminus K\,.
\end{equation*}
The function $\omega_\Sigma$ generates a smooth rotationally symmetric disconnected hypersurface
$F_\Sigma:\Sigma\rightarrow\R^{n+1}$, where $\Sigma$ is the disjoint union of $n_p+1$ cylinders.
We term the support hypersurface $F_\Sigma$ a \emph{pinching cylinder}.
\label{pinchingcylinder}
\end{defn}

Note that if $n_p=0$ in Definition \ref{pinchingcylinder} we are in one of the cases considered earlier in the paper.

\begin{rmk}
Although we require that $\omega_\Sigma$ be only continuous on $\R$, it may pinch and be smooth (or analytic) everywhere
on $\R$. For example, this is the case if $\omega_\Sigma$ is a non-negative polynomial in $z$ with zeros; for example,
\[
\omega_\Sigma(z) = (z-2)^2(z+2)^2\,.
\]
\end{rmk}

\begin{thm}[Flow in conical pinching cylinders]
Let $\Sigma$ be a pinching cylinder as in Definition \ref{pinchingcylinder} with $n_p=1$.
Let $w_1 = z^* = 0$.
Assume \eqref{Sigma_graph} (understood as limits from the left and right at points in $P$).
Suppose that for all $z\in\R\setminus\{0\}$,
\begin{equation}
\label{conelike}
z\frac{d\omega_\Sigma}{dz}(z) > 0\,.
\end{equation}
Then the maximal time of existence for any solution $\omega:D(t)\times[0,T)\rightarrow\R$ to \eqref{Neumannproblem}
satisfies $T < \infty$.
The hypersurfaces $F:D^n\times[0,T)\rightarrow\R^{n+1}$ generated by $\omega$ contract as $t\rightarrow T$ to the point
$(0,z^*)$.
\label{thmsingularities}
\end{thm}
\begin{proof}
The proof height and gradient estimates goes through exactly as in Theorem \ref{LTE} and Lemma \ref{gradientestimates}.
We do not have global existence however, as in this setting, we do not have a uniform bound on $r(t)$ from below.
We claim that the solution $\omega:D(t)\times[0,T)\rightarrow\R$ of \eqref{Neumannproblem} exists smoothly
for all $t\in[0,T)$, $T<\infty$, and $r(t)\rightarrow0$ as $t\rightarrow0$.

To see this, we first show $T<\infty$.
For the sake of contradiction, assume that the graphs exist for all time, that is, $T=\infty$.
Then the solutions converge to a flat disck perpendicular to the contact hypersurface $\Sigma$ as per Theorem
\eqref{thmconvergence}.
However, any flat disk must be supported on $\Sigma$ by a point where the gradient of $\omega_\Sigma$ vanishes.
Such a point (by \eqref{conelike}) does not exist.

Therefore $T< \infty$.
Since the height and gradient bounds provides us with uniform $C^1$ estimates for all time, the only posibility
preventing global existence is that $r(t)\rightarrow0$ as $t\rightarrow T$.
Therefore the solution converges to a point on the axis of rotation.
The solution however must also satisfy the Neumann condition, and so the limit point must be a point where $\Sigma$
pinches off; as there is only one such point where this occurs, we are done.
\end{proof}

\begin{rmk}[Non-rotational initial data]
Any initially bounded mean curvature flow with free boundary, irrespective of
symmetry or topological properties, exists at most for finite time when
supported on a pinching cylinder as in Theorem \ref{thmsingularities}.
This is because so long as the initial immersion is bounded, we may always
construct a rotationally symmetric graphical solution such that the initial
immersion lies between this solution and the pinchoff point $(0,z^*)$.
The flow generated by this pair of initial data remain disjoint by the
comparison principle, and as the rotationally symmetric solution contracts to a
point in finite time, the flow of immersions must either develop a curvature
singularity in finite time or contract to the same point (and possibly remain
regular while doing so).

Similarly, in a shrinking neck region, we may use the rotationally symmetric
graphical solutions as barriers to obtain that any mean curvature flow with
free boundary whose initial data is contained in a shrinking neck either exists
for all time and converges to a flat disk or develops a curvature singularity
in finite time.
\end{rmk}

Our next task is to determine the type of the singularity.
We are able to show that in most cases the singularity is Type I or better
(Type 0: that it is not a curvature singularity at all but a loss of domain).
The cases that allow us to do this are when the gradient of $\omega_\Sigma$ is bounded.
This includes cones and cusps.
We are not yet able to conclude the same for the case of parabolae, that is, when the gradient of $\omega_\Sigma$
is unbounded on $Oz\setminus P$.

\begin{defn}[Singularities]
Let $F:D^n\times[0,T)\rightarrow\R^{n+1}$ be a mean curvature flow with free boundary supported on a pinching cylinder.
If there exists an $\varepsilon>0$ such that for all $t\in(T-\varepsilon,T)$
\begin{itemize}
\item the second fundamental form is uniformly bounded, that is,
\[
|A|^2(x,t) \le C < \infty\,,
\]
then we say the singularity is Type 0;
\item the second fundamental form is uniformly controlled under parabolic rescaling, that is,
\[
|A|^2(x,t) \le \frac{C}{T-t}\,,
\]
then we say the singularity is Type I;
\item neither of the previous two cases apply, we say the singularity is Type II.
\end{itemize}
\end{defn}

\begin{thm}[Type 1 singularities]
Let $\omega_\Sigma$ and $\omega_0$ be as in Theorem \ref{thmsingularities}.
If there exist two constants $0<C_1<\infty$ and $C_2<\infty$ such that for $z$
sufficiently close to $z^*$ we have:
\begin{itemize}
\item Conical pinchoff
\begin{align*}
C_1 \le \bigg|\frac{d\omega_\Sigma}{dz}(z^*)\bigg| \le C_2,
\end{align*}
then the singularity from Theorem \ref{thmsingularities} is Type I;
\item Polynomial pinchoff
\begin{align*}
C_1|\omega_\Sigma(z)|^{\sigma} \le \bigg|\frac{d\omega_\Sigma}{dz}(z)\bigg| \le C_2|\omega_\Sigma(z)|^{\sigma}\,,
\end{align*}
for $\sigma<1$, then the singularity from Theorem \ref{thmsingularities} is
Type I, and in particular there exist $\hat{C}_1, \hat{C}_2$ such that for $t$
sufficiently close to $T$ we have
\[
\frac{\hat{C}_1}{T-t}
\le
|A|^2(x,t)
\le
\frac{\hat{C}_2}{T-t}
\]
\end{itemize}
\label{thmtypeI}
\end{thm}

Before starting the proof of the theorem we need to compute the norm squared of
the second fundamental form and mean curvature in terms of the profile curve
$\omega$.

\begin{lem}
For a rotationally symmetric hypersurface generated by the rotation of a graph
function $\omega$ about an axis perpendicular to the graph direction, the norm
squared of the second fundamental form and mean curvature are given by the
formulae
\begin{align*}
|A|^2 &= \frac{1}{(1+(\frac{d\omega}{dy})^2)^3} \Big(\frac{d^2\omega}{dy^2}\Big)^2
          + \frac{1}{1+(\frac{d\omega}{dy})^2}\frac{1}{y^2}\Big(\frac{d\omega}{dy}\Big)^2\notag\\
H &= -\frac{1}{\sqrt{1+(\frac{d\omega}{dy})^2}^3} \frac{d^2\omega}{dy^2}\ -\ \frac{1}{\sqrt{1+(\frac{d\omega}{dy})^2}}\frac{1}{y}\frac{d\omega}{dy}.
\end{align*}
\label{secondff}
\end{lem}
\begin{proof}
The proof is a lengthy but straightforward computation using the parametrisation for a rotationally symmetric graph,
that is, $F(x,t)=(x,\omega(|x|,t))$, where $x\in D^n$, and denoting $y=|x|$.
\end{proof}

\begin{proof}[Proof of Theorem \ref{thmtypeI}]
Given that the gradient of $\omega_\Sigma$ is uniformly bounded, as before in the proof of Theorem \ref{LTE} and the
proof of Theorem \ref{thmsingularities} we have that the $\omega$ satisfy uniform $C^1$ estimates up to the time of
singularity.
Let us denote this time by $T$. The estimates imply that there exists a constant $C_4<\infty$ depending only on the
initial data such that
\begin{align*}
\bigg|\frac{d^2\omega}{dy^2}\bigg| \leq C_4\,.
\end{align*}

Thus the second fundamental form will explode at worst as quickly as $r(t)\searrow0$, that is, there exists a
constant $C=C(C_2,C_4)<\infty$ such that
\begin{align*}
|A|^2(y,t) \le C\frac{1}{y^2}\bigg(\frac{d\omega}{dy}\bigg)^2(y,t),
\end{align*}
for all $t$ and $y$.

On the rotation boundary, that is at $y=0$, the right hand side is uniformly bounded by symmetry.
(A unique tangent plane exists at the origin.)
Everywhere else the gradient and $1/y$ is bounded by an absolute constant multiplied by it's value at the boundary.

Thus there exists a constant denoted by
abuse of notation $C=C(C,C_2,C_4)<\infty$ such that
\begin{align}
\sup_{[0,r(t)]}|A|^2(t)\ \leq\ C \frac{1}{r(t)^2}\bigg(\frac{d\omega}{dy}\bigg)^2(r(t),t)\,,
\label{estbu1}
\end{align}
for all $t$.
From here we separate the proof into the two cases.
First assume the case of cones, that is there exists a second constant $C_1>0$ such that
\begin{align*}
\bigg|\frac{d\omega_\Sigma}{dz}(z^*)\bigg| \ge C_1.
\end{align*}

From our estimates above, there exists a constant denoted by abuse of notation $C=C(C,C_2,C_4)<\infty$ such that
\begin{align}
\sup_{[0,r(t)]}|A|^2(t)\ \leq\ C \frac{1}{r(t)^2},
\label{first2}
\end{align}
for all $t$.
To compute the rate of blow up for the boundary point $r(t)$, we use the time evolution
for $r(t)$ (computed earlier in the proof of Theorem \ref{minimumflatdiscsH})
\begin{align*}
r'(t)\ =\ -\frac{H}{v}\frac{d\omega_\Sigma}{dz}.
\end{align*}
Substituting for the Neumann boundary condition \eqref{Neumanncondition} and the formula for the mean curvature in Lemma
\ref{secondff}, we obtain
\begin{align*}
r'(t)r(t)\ =\ -\frac{(\frac{d\omega}{dy})^2}{1+(\frac{d\omega}{dy})^2}\ -\ \frac{\frac{d^2\omega}{dy^2}\frac{d\omega}{dy}}{(1+(\frac{d\omega}{dy})^2)^2} r(t).
\end{align*}
Given that the gradient is bounded away from $0$ by $C_1$ (using the Neumann condition and the bound on the gradient of
$\omega_\Sigma$), and also bounded from above by $C_2$, we have that
\begin{align*}
r'(t)r(t)\ \leq\ -\frac{C_1^2}{1+C_2^2}\ -\ \frac{\frac{d^2\omega}{dy^2}\frac{d\omega}{dy}}{(1+(\frac{d\omega}{dy})^2)^2} r(t).
\end{align*}
We know $r(t)\rightarrow 0$ as $t\rightarrow T$ where $T$ is the final time of existence. We also know that the second
derivative is bounded by $C_4$. Thus we can choose $0<t^*<T$, independently of the sign of the second term above in
$r'(t)r(t)$, such that
\begin{align*}
r'(t)r(t)\ \leq\ -C_5.
\end{align*}
for some constant $0<C_5=C_5(C_1,C_2,C_4)$ for all $t^*<t<T$.
Note that $C_5$ is bounded away from $0$ is independent of $t$.
Integrating from $t<T$ to $T$ and using the fact that $r(T)=0$ we find
\begin{align*}
r^2(t)\ \geq\ 2C_5(T-t),
\end{align*}
for all $t\geq t^*$.
Substituting this into \eqref{first2} we obtain the following bound for the second fundamental form
\begin{align*}
\sup_{[0,r(t)]}|A|^2(t)\ \leq\ \frac{C}{2C_5}\frac{1}{T-t}
\end{align*}
for all $t\in (t^*,T)$, that is, the singularity is Type I.

Now consider the case of polynomial pinchoff: for $z$ sufficiently close to $z^*$ we have
\begin{align*}
C_1|\omega_\Sigma(z)|^{\sigma} \le \bigg|\frac{d\omega_\Sigma}{dz}(z)\bigg| \le C_2|\omega_\Sigma(z)|^{\sigma}\,.
\end{align*}
Using $r'(t) = -\frac{H}{v}\omega_\Sigma'(\omega(r(t),t))$ and the above we estimate:
\begin{align*}
r^{1-2\sigma}(t)r'(t)
\le Cr^{1-2\sigma}(t)-\frac{C_1}{1+C_2r^{2\sigma}}
\le -C_1/2
\end{align*}
for $t$ sufficiently close to $T$ (since $\sigma < 1$) and some $C=C(C_1,C_2,C_4)$ , and so
\begin{align*}
-r^{2-2\sigma}(t)
 = \int_t^T (r^{2-2\sigma}(t))'\,ds
\le \int_t^T -C_1(1-\sigma)\,ds
= -C_1(1-\sigma)(T-t)
\,.
\end{align*}
This implies
\[
\frac{1}{r^{2-2\sigma}(t)} \le \frac{1}{C_1(2-2\sigma)}\frac{1}{T-t}\,.
\]
Estimating as above (beginning at estimate \eqref{estbu1} earlier) we find
\begin{align*}
\sup_{[0,r(t)]}|A|^2(t)
 &\le C\frac{1}{r^2(t)}\bigg(\frac{d\omega}{dy}\bigg)^2(r(t),t)
\\
 &= C\frac{1}{r^2(t)}\bigg(\frac{d\omega_\Sigma}{dz}\bigg)^2(\omega(r(t),t))
\\
 &\le C\frac{1}{r^{2-2\sigma}(t)}
\\
 &\le C(\sigma) \frac{1}{T-t}
\,.
\end{align*}
Therefore the singularity is Type I.
Now as the assumption is two-sided, we find that (for a different constant
$C(\sigma)$) the same estimate above for the second fundamental form holds, but
from below.
Therefore the singularity is no better and no worse than Type I, and the
statement follows.

\end{proof}

\begin{rmk}
Conical pinchoff is a special case of polynomial pinchoff.
For polynomial pinchoff, it isn't possible to satisfy all condition of the theorem for $\sigma\ge1$.
For $\sigma>0$, the pinchoff is \emph{convex} and for $\sigma<0$ the pinchoff is \emph{concave}.
These names come from the following examples:
\[
\omega_\Sigma(z) = z^\alpha
\]
satisfies $\omega_\Sigma'(z) = \alpha \omega_\Sigma^{1-\frac1\alpha}(z)$.
Therefore $\alpha > 1$ corresponds to $\sigma \in (0,1)$ and $\alpha < 1$ corresponds
to $\sigma < 0$.
Clearly all asymptotically polynomial pinchoffs are allowed by the condition $\sigma < 1$.
Concave pinchoff is related to the singularity resulting from mean curvature
flow with free boundary supported in the sphere, studied by Stahl
\cite{stahl1996res}.
\end{rmk}

\begin{thm}[Type 0 singularities]
Let $\omega_\Sigma:Oz\rightarrow\R$ be the profile curve of a rotationally
symmetric hypersurface satisfying \eqref{Sigma_graph} and
\[
\lim_{z\rightarrow\infty}\omega_\Sigma(z) = 0\,,\quad
\bigg|\frac{d\omega_\Sigma}{dz}(z)\bigg| \le C|\omega_\Sigma|^{1+\sigma}(z)\,,\quad \sigma>0\,.
\]
Then the maximal time of existence for any solution $\omega:D(t)\times[0,T)\rightarrow\R$ to \eqref{Neumannproblem}
satisfies $T = \infty$.
The hypersurfaces $F:D^n\times[0,T)\rightarrow\R^{n+1}$ generated by $\omega$ satisfy
\[
||A||_\infty^2(t) \rightarrow \alpha_0\quad\text{as $t\rightarrow\infty$}\,,
\]
and so either
\begin{itemize}
\item $F(D^n,t)$ converges smoothly to a flat disk; or
\item Modulo translation, $F(D^n,t)$ converges to a flat point, that is, a singularity of Type 0.
\end{itemize}
\label{thmtype0}
\end{thm}
\begin{proof}
First, a uniform a-priori gradient bound follows by applying Lemma
\ref{gradientestimates}.
Note that the difference
\[
\sup\{|\omega(y_1,t) - \omega(y_2,t)|\,:\,y_1,y_2\in[0,r(t)]\}
\]
is uniformly bounded, since if it weren't, this would contradict the uniform
gradient bound.
Therefore, the translated flow $\hat\omega(y,t) := \omega(y,t) - \omega(0,t)$
has uniformly bounded height, and so, exists for all time.
Note importantly that the domain of $\hat\omega$ is equal to the domain of
$\omega$, that is, $r(t)$ is invariant under translation.

For the original solution, we have $T = \infty$ and global existence, however the
height may become unbounded.

We calculate, as in the proof of Theorem \ref{thmtypeI} above,
\begin{align*}
\sup_{[0,r(t)]}|A|^2(t)
 &\le C\frac{1}{r^2(t)}\bigg(\frac{d\omega}{dy}\bigg)^2(r(t),t)
\\
 &\le C\frac{1}{r^2(t)}\bigg(\frac{d\omega_\Sigma}{dz}\bigg)^2(\omega(r(t),t))
\\
 &\le C\frac{|\omega_\Sigma|^{2+2\sigma}(r(t),t)}{r^2(t)}
\\
 &\le Cr^{2\sigma}(t)\,.
\end{align*}
Therefore the claim follows with $\displaystyle \alpha_0 = C\lim_{t\rightarrow\infty}r^{2\sigma}(t)$.
If $r(t)\rightarrow0$ and $|\omega(0,t)|\rightarrow\infty$ then the limit is a
flat point, and if $r(t)\rightarrow r_\infty>0$ then the proof of Theorem
\ref{thmconvergence} applies and the limit is a flat disk.
\end{proof}

\begin{rmk}
Examples of support hypersurfaces with profile curves satsifying the conditions
of Theorem \ref{thmtype0} include exponentials and reciprocal polynomials, such as
\[
\omega_\Sigma(z) = e^{-z}
\]
and a (monotone) mollification of
\[
\omega_\Sigma(z) =
\begin{cases}
\frac1z\,,\quad\text{for }z>1
\\
-z+2\,,\quad\text{for }z\le1\,.
\end{cases}
\]
\end{rmk}

\section{Uniqueness results for minimal hypersurfaces with free boundary}

In this section we apply the parabolic results proved earlier to the uniqueness
problem for minimal hypersurfaces with free boundary.
We emphasize that the results in this section hold for immersed minimal
hypersurfaces with free boundary, that is, without any restrictions on topology, symmetry,
or graphicality.

In this section we assume $\Sigma$ to be a pinching oscillating cylinder, as in Section 2.3.
Examples of this include catenoids, unduloids, cones, parabolae, and so on.

Generically, an oscillating cylinder decomposes into belly regions and
shrinking neck regions (if maximal, these have non-trivial overlap).
In belly regions, there may exist flat minimal disks that are not rotationally
symmetric with respect to the $Oz$ axis; for example, if part of the belly
region is spherical, then there exist infinitely many such tilted flat disks.
Clearly these disks serve as barriers for the mean curvature flow with free boundary.
Unfortunately, there does not exist a mean curvature flow with free boundary
that is asymptotic (in positive time) to such slanted disks.

For the case of a shrinking neck region however, solutions are asymptotic to
flat disks, and these disks have $Oz$ as their axis of rotation.
Our result is the following:

\begin{thm}[Uniqueness in shrinking necks]
Let $F:M^n\rightarrow\R^{n+1}$ be an immersed bounded smooth minimal
hypersurface with free boundary on a pinching cylinder
$F_\Sigma:\Sigma\rightarrow\R^{n+1}$.
If $F(M)\subset\Theta(z_1,z_2)$ where $\Theta(z_1,z_2)$ is a shrinking neck
region, then $F(M)$ is a standard flat disk.
\label{thmuniqueness}
\end{thm}
\begin{proof}
We squeeze the minimal hypersurface between two rotationally symmetric graphical
solutions to the mean curvature flow with free boundary.
First, consider the maximal shrinking neck region $\Theta(z_{min},z_{max}) \supset \Theta(z_1,z_2)$.
Suppose that $\Theta(z_{min},z_{max})$ is bounded.
Then by maximality there exists flat minimal disks supported on $F_\Sigma$ at $z_{min}$ and $z_{max}$.
As $F(M)\subset\Theta(z_1,z_2)\subset\Theta(z_{min},z_{max})$, there exist
graphical rotationally symmetric smooth hypersurfaces
$f^1,f^2:D^n\rightarrow\R^{n+1}$ supported on $F_\Sigma$ and disjoint from the
minimal disks at $z_{min}$, $z_{max}$, and $F(M)$ such that $f^1 < F(M) <
f^2$.\footnote{We say a hypersurface $M$ is less than a hypersurface $N$ in
$\Theta(a_1,a_2)$ if along each vertical line from $\{z=a_1\}$ to $\{z=a_2\}$,
the intersection point with $M$ is lower than the intersection point with $N$.}

Now take $f^1$ and $f^2$ as initial data for the mean curvature flow with free boundary,
generating flows $F^1,F^2:D^n\times[0,\infty)\rightarrow\R^{n+1}$.
By the results of Section 2, each of these flows converge to minimal disks.
By the comparison principle (see for example \cite{thesisstahl,thesisvulcanov})
we have that the hypersurfaces $F^1(D^n,t)$, $F^2(D^n,t)$ are disjoint from
each other, as well as disjoint from $F(M)$.
This is because if they were to intersect at any point, it would be a point of
tangency, and then, as $F(M)$ is minimal and $F^i(D^n,t)$ is not minimal for
all $t\in[0,\infty)$, this would be a contradiction. (This is the only part of
the argument where we require any smoothness of the minimal immersion $F$.)

Now the region these flows foliate is
\[
\Omega = \bigcup \Big\{ F^i(D^n,t)\,:\,t\in[0,\infty)\Big\}\,.
\]
The minimal hypersurface $F(M)$ must be disjoint from this region, and it may
not lie above $f^2$ or below $f^1$ by construction.
Therefore, as $F(M)$ lies in a shrinking neck region $\Theta(z_1,z_2)$, it is supported in a purely cylindrical portion of $\Theta(z_1,z_2)$.

We extend our foliation through this cylindrical region by translation, that is, we take one further flow $g:D^n\times[0,1]\rightarrow\R^{n+1}$ where
\[
\partial_tg = \nu\,,\quad
g(D^n,0) = \lim_{t\rightarrow\infty} F^1(D^n,t)\,,\quad
g(D^n,1) =  \lim_{t\rightarrow\infty} F^2(D^n,t)\,.
\]
This flow completes the foliation of the region containing $F(M)$.
The flow $g$ is a flow of minimal hypersurfaces.
Therefore, at a first point and time $t^*\in[0,1]$ of tangency, we must have $F(M) = g(D^n,t^*)$, that is, $F(M)$ is a flat disk.

Suppose now that $\Theta(z_{min},z_{max})$ is unbounded on one or both sides; say $z_1 = -\infty$.
Now the boundedness hypothesis on $F(M)$ implies that there exists a
rotationally symmetric graphical smooth hypersurface
$f^1:D^n\rightarrow\R^{n+1}$ supported on $F_\Sigma$ and disjoint from $F(M)$
such that $f^1 < F(M)$.
We can use this as initial data and proceed using the argument above.
Similarly, if $z_2 = \infty$, boundedness of $F(M)$ implies that there exists a
rotationally symmetric graphical smooth hypersurface
$f^2:D^n\rightarrow\R^{n+1}$ supported on $F_\Sigma$ and disjoint from $F(M)$
such that $F(M) < f^2$.
Again, we can use this $f^2$ in the argument above.

In all cases $F(M)$ is a flat disk, and so we are finished.
\end{proof}

This theorem implies in particular:

\begin{cor}[Non-existence of immersed minimal free boundary hypersurface with
topological type other than that of a disk]
There exists no smooth bounded immersed minimal $n$-dimensional hypersurface
supported on $F_\Sigma$ in any topology other than that of the flat disk.
\label{thmnonexistence1}
\end{cor}

\begin{cor}[The catenoid case]
The only bounded smooth immersed minimal hypersurface with free boundary on a
catenoid is the flat disk supported at the origin.
\end{cor}

Similar results are provable using the same method as above.
We have not attempted to give an exhaustive list.
When there is no minimal disk supported on $F_\Sigma$, then one may guess that
there is no minimal hypersurface supported on $F_\Sigma$.
One situation where this holds is the following:

\begin{prop}[Non-existence of minimal hypersurfaces in cones, parabolae]
Let $F_\Sigma$ be as in Theorem \eqref{thmsingularities}.
There does not exist an immersed bounded smooth minimal hypersurface
$F:M^n\rightarrow\R^{n+1}$ supported on $F_\Sigma$.
\label{thmnonexistence3}
\end{prop}
\begin{proof}
As the proof is similar to that of Theorem 3.1, we only give a brief outline.
In Theorem \ref{thmsingularities}, there exists precisely one point of pinching
for $F_\Sigma$.
Without loss of generality we can assume that $F_\Sigma$ lies either above or
below this point of pinching.
In either case, we can construct a rotationally symmetric graphical smooth
hypersurface $f:D^n\rightarrow\R^{n+1}$ supported on $F_\Sigma$ such that $F(M)$ lies between $f$ and the pinching point of $\Sigma$.
We use $f$ to generate a mean curvature flow with free boundary.
This yields a foliation of the region between $f$ and the pinching point, which
must by smoothness have a first point and time of tangency with $F$.
However $F$ is minimal and the flow generated by $f$ is not; this yields a
contradiction.
\end{proof}

\section*{acknowledgements}

The author is supported by Australian Research Council Discovery grant DP150100375 at
the University of Wollongong. The author is grateful to the Korea Institute
for Advanced Study and Hojoo Lee for his hospitality and interesting discussions related to this work.

\bibliographystyle{plain}
\bibliography{mbib}

\end{document}